\documentclass{amsart} \usepackage{latexsym} \usepackage{amssymb}
\usepackage{amsmath} \usepackage{enumerate}
\usepackage[dvips]{graphicx} \usepackage[latin1]{inputenc}
\title[Non-existence of acip for exponential maps]{Non-existence of
  absolutely continuous invariant probabilities for exponential maps }

\author{Neil Dobbs}\address{Institute of Mathematics of the Polish
  Academy of Sciences, ul. {\'S}niadeckich 8 P.O. Box 21, 00-956
  Warszawa, Poland} 
\email{neil.dobbs@gmail.com}

\author{Bart{\l}omiej Skorulski} \address{
  Departamento de Matem{\'a}ticas, Universidad Cat{\'o}lica del Norte,
  Avenida Angamos 0610, Casilla 1280, Antofagasta, Chile}
\email{bskorulski@ucn.cl}
\date{January 30, 2008}

\thanks{The authors were partially supported by Research Network on Low
  Dimensional Dynamics, PBCT ACT 17, CONICYT, Chile and by the EU Research Training Network ``Conformal Structures and Dynamics''. The second author
  was also supported by Chilean FONDECYT Grant No. 11060538.}

\newcommand\eprf{\hfill$\Box$}
\newcommand\Swan{{\'S}wi{\c{a}}tek }

\newcommand\arr{\mathbb{R}}
\newcommand\ccc{\mathbb{C}}

\newcommand\re{\mathrm{Re}}

\newcommand\scrP{\mathcal{P}}

\newcommand\cbar{{\overline{\mathbb{C}}}}

\newtheorem{thm}{Theorem}

\newtheorem{lem}[thm]{Lemma}

\begin{document}

\maketitle
\begin{abstract}
  We show that for entire maps of the form $z \mapsto \lambda \exp(z)$
  such that the orbit of zero is bounded and such that Lebesgue almost
  every point is transitive, no absolutely continuous invariant
  probability measure can exist. This answers a long-standing open
  problem. 
\end{abstract}

In this paper we introduce a new method to deal with the problem of
existence of invariant measures for entire maps. To illustrate this
method, avoiding uninteresting technical difficulties, we show the
following theorem.
\begin{thm} \label{thm:easy}Let $\lambda\in\ccc\setminus\{0\}$ be such
  that the Julia set of $f:z\mapsto \lambda \exp(z)$ is equal to $\ccc$,
  the forward orbit of $0$ is bounded and such that there is a set of
  positive Lebesgue measure of points $z\in \ccc$ such that
  $\omega(z)\not\subset \scrP(f)$. Then $f$ has a $\sigma$-finite
  absolutely continuous invariant measure, but it does not have an
  absolutely continuous invariant probability measure.
\end{thm}
We denote, as usual, the $\omega$-limit set of $z\in \ccc$ by $\omega(z) \subset \cbar$
 and the post-singular set by $\scrP(f)$, here equal to the union of $\{\infty\}$ with the closure of the orbit of $0$.

Theorem~\ref{thm:easy} implies, in particular, that the map $z\mapsto
2\pi i \exp(z)$ has no absolutely continuous invariant probability
measure, which was a long-standing open problem (see
\cite{MayUrb07}). 

For a class of unimodal maps $f$ of the interval with non-recurrent critical point, Benedicks and Misiurewicz in \cite{MisBen:Flat}  showed that there exists an absolutely continuous invariant probability measure if and only if $\int \log |f'(x)| dx > -\infty$. The necessity of the integrability condition was later extended by the first author (in Theorem~1 of \cite{Me:Cusp})  to all $C^{1+\epsilon}$ interval maps without any hypothesis on the critical orbits, but under the assumption that the measure has positive Lyapunov exponent. As an example, no unimodal map of the form $x \mapsto C_1 - C_2\exp(-|x|^{-\alpha})$ with $\alpha \geq 1$ has an absolutely continuous invariant probability measure with positive Lyapunov exponent. 

This paper extends the main result of \cite{MisBen:Flat}  to the holomorphic setting. The strategy of the proof has two elements: construct a \emph{nice set} on which the density of a hypothetical measure must be bounded away from zero; show that the return time to the nice set is not integrable. This is similar in philosophy to the proof of Benedicks and Misiurewicz, although the settings differ.

 Existence of a
$\sigma$-finite measure is not new; it was shown in \cite{KotUrb03} under weaker
hypotheses but with a considerably more difficult proof.
On the other hand, existence of absolutely continuous invariant probability measures for
transcendental entire maps has been an interesting and open question
for some time with a response in only one situation: J-M Hemke in \cite{Hem05} 
proved that, for a class of entire maps for which the orbits of all asymptotic values converge to infinity sufficiently fast, $\omega(z) \subset \scrP(f)$. For these maps, $\scrP(f)$ has zero measure and it follows from the Poincaré Recurrence Theorem that no absolutely continuous invariant probability measure can exist. Hemke's work 
generalised a result proved independently by  M.\ Rees and M.\ Lyubich for $z \mapsto \exp(z)$ (\cite{Rees86}, \cite{Lyu87b}). 

 For non-entire maps the second author, in \cite{Sko03}, has
another negative result for some postcritically finite tangent
maps. For a large class of transcendental non-entire maps which
satisfy a Misiurewicz-type condition J.\ Kotus and G.\ \Swan in
\cite{KotSwi07} showed that absolutely continuous invariant
probability measures can exist.

The mathematics involved in the proof have the merit of being
surprisingly elementary. An important and somewhat magical technique
is Juan Rivera-Letelier's construction of nice sets for rational
dynamics (see \cite{Riv07}) which we adapt to the entire setting. 

An open set $U$ is called \emph{nice} if $f^n(\partial U) \cap U =
\emptyset$ for all $n >0$. This implies that every pair of pullbacks
(connected components of $f^{-n}(U), f^{-n'}(U)$ for some $n,n'\geq 0$)
is either nested or disjoint.
  Let us fix some $D >0$ such that $\scrP(f) \subset B(0,D)$. 

\begin{lem} \label{lem:magic} For each sufficiently large $x \in
  \arr$, there exists a  connected nice set $U \subset
  \ccc$ satisfying $B(x,4\pi) \subset U \subset B(x, 8\pi)$.
\end{lem}
\begin{proof}
  There exists a $K >1$ such that for any $r$ and any holomorphic
  function $g$, univalent on $B(x, Kr)$, one has
  $$
  \left|\frac{g'(z)}{g'(z')}\right| \leq 2
  $$
  for all $z,z' \in B(x,r)$, by the Koebe distortion theorem.

Let $x$ satisfy $x >
  8K\pi + D$. Let $W$ be a (connected) pullback of $B(x,8\pi)$ and let $n > 0$ be such that 
  $f^n$ maps $W$ univalently onto $B(x,8\pi)$. Since
  $f^n_{|W}$ extends to map univalently onto $B(x, 8K\pi)$, it follows
  that the distortion of $f^n$ restricted to $W$ is bounded by 2.
Thus
  there is $r>0$ such that $B((f^n_{|W})^{-1}(x),r)\subset
  W\subset B((f^n_{|W})^{-1}(x),2r)$. But $B((f^n_{|W})^{-1}(x),r)$ must lie
  in a horizontal strip of height $2\pi$, so we have that $|W| <
  4\pi$.

  We shall use this to construct nice sets exactly as per
  \cite{Riv07}. We include the proof for the reader's convenience: Let
  $U_0 := B(x,4\pi)$ and define $U_n$ as the connected component of
  $\bigcup_{i=0}^n f^{-i}(U_0)$ containing $U_0$ and $U = \bigcup_{n \geq 0} U_n$. We
  prove by induction that $U_n \subset B(x,8\pi)$ for all $n \geq
  0$. This is clearly true for $n = 0$. So suppose it is true for all
  $n \leq k$. We must show it holds for $n = k+1$.

  Let $X$ be a connected component of $U_{k+1} \setminus U_0$. Then
  there is a minimal $m \geq 0$ such that $f^m(z) \in U_0$ for some $z
  \in X$, and necessarily $m \geq 1$. Consider $f^m(X)$. This set is
  contained in $U_{k+1-m}$, and so by hypothesis is contained in
  $B(x,8\pi)$. But then $X$, being connected, is contained in some
  pullback $W$ with $|W| < 4\pi$. The result follows.
\end{proof}

\begin{lem} \label{lem:escape} There exists a $c > 0$ such that if
  $f^n(z) \notin B(0,2D)$ then $n > -c\log |z|$.
\end{lem}
\begin{proof}
  Let $M > 1$ be such that $|f'(z)| < M$ for all $z \in
  B(0,2D)$. Suppose $f^n(z) \notin B(0,2D)$. Then $|f^n(z) - f^n(0)| >
  D > 1$. This implies that $|z - 0| = |z| > M^{-n}$. Thus $\log |z| >
  -n \log M$ and $n > (-1/\log M) \log |z|$.  \eprf

  In what follows, let $U$ be a nice set given by Lemma
  \ref{lem:magic} for some $x > 8\pi + 2D$, we fix $x$ too. In
  particular, $U \cap B(0,2D) = \emptyset$. We denote by $r_U(z)$ the
  first return time of $z$ to $U$. Also let $r, \phi \in \arr$ such
  that $\lambda = re^{i\phi}$.
\end{proof}

\begin{lem} \label{lem:longtime}
  There exists $C \in \arr$  and $c > 0$ with the following property. Suppose $z \in U$ and $\re(f^k(z)) \leq  -K $ for some $0 < k < r_U(z)$ and $K > 0$. Then $r_U(z) > C + cK$.
\end{lem}
\begin{proof}
  Let $c$ be given by Lemma \ref{lem:escape}.  We have $|f^{k+1}(z)|
  \leq r e^{-K}$. Then the time it takes for $f^{k+1}(z)$ to leave
  $B(0,2D)$ is greater than $-c \log(r e^{-K}) = -c(-K + \log r)$ by
  Lemma \ref{lem:escape}. Take $C := -c\log r$.
\end{proof}

\begin{lem} \label{lem:infty}
  Denote by $m$ Lebesgue measure. Then $\int_U r_U(z) dm = \infty$.
\end{lem}
\begin{proof}
  Define $h : \arr \to \arr$ by $h(y) = (r/2)\exp(y)$ and let
  $$
  S_R := \{z: \re(z) > x \mbox{ and } \arg f(z) \in [-\pi/4, \pi/4] \}
  $$
  and
  $$
  S_L := \{z: \re(z) > x \mbox{ and } \arg f(z) \in [3\pi/4, 5\pi/4]  \}.
  $$
Note that each connected component of $\{z : \arg f(z) \in [-\pi/4, \pi/4]\}$ is a horizontal strip of height $\pi/2$ and the components are periodic of period $2 i \pi$. A similar statement holds for $\{z : \arg f(z)\in [3\pi/4, 5\pi/4] \}$.
  Let 
$$P_n := \{z \in B(x,4\pi) : f^k(z) \in S_R \mbox{ for all } 0
  \leq k \leq n\},$$
 and let $Q_n := P_{n-1} \cap f^{-n}(S_L)$. 
   For $z \in S_R$, $\re(f(z)) \geq h(\re(z))$, so  by induction, for all $z \in P_n$, $\re(f^n(z)) \geq
  h^n(x)$. 
Then distortion arguments
  like in \cite{Mcm87} give that $m(Q_n)/m(P_n)$ tends to one and that
  $$
  \lim_{n\to \infty} m(P_n)/m(P_{n+1}) = 1/4.
  $$
  Thus there exists a $\gamma \in (0,1/4)$ such that for all $n \geq 1$,
  $$
  m(Q_n) \geq \gamma^n.$$ Now for $z \in Q_n$, $\re(f^{n+1}(z)) <
  -h^{n+1}(x)$, so we have $r_U(z) > C + ch^{n+1}(x)$ where the
  constants $c,C$ are given by Lemma \ref{lem:longtime}. But $h^n(x)$
  grows with $n$ faster than any exponential so
  $$
  \lim_{n\to \infty} m(Q_n) \inf \{r_U(z) : z \in Q_n\} = \infty.$$
\end{proof}

\begin{proof}[Proof of Theorem \ref{thm:easy}]
  Let $\psi$ denote the first return map to $U$. Since $U$ is nice and
  disjoint from $\scrP(f)$, every connected component of the domain of
  $\psi$ is mapped univalently onto $U$ by $\psi$.  Moreover the branches
  of $\psi$ are uniformly extendible, so the Koebe distortion theorem
  gives a uniform distortion bound for all branches of all iterates of
  $\psi$. Note that by Corollary 2\ of \cite{Boc96}, Lebesgue
  almost every point has a transitive orbit. 

Denote by $V_k$ the set of points from $U$ whose first return time
to $U$ is $k$. By the Folklore Theorem (see for example \cite{Gou04}), 
there exists a unique absolutely continuous invariant
probability $\nu$ for $\psi$ and its density is bounded below by some
$\varepsilon > 0$. Then the spread measure $\mu:=\sum_{k=1}^\infty
\sum_{n=0}^{k-1} f_*^n\nu_{|V_k}$ is a $\sigma$-finite absolutely
continuous invariant measure for $f$. This gives the easy proof of its
existence.

  Suppose now $\mu$ is an absolutely continuous $f$-invariant
  probability measure. 
   By transitivity of
  Lebesgue almost every point, $\mu(U) > 0$. Then  $\mu$ is also a finite
  invariant measure for $\psi$, since $\psi$ is a first return
  map. By uniqueness, the density of $\mu$ is  then bounded from below on $U$ by
   $\mu(U) \varepsilon >0$.  Thus
  $$
  1 = \int_U r_U(z) \, d\mu \geq \mu(U) \varepsilon \int_U r_U(z) \, dm,
  $$
  the first equality being Kac' Lemma. This contradicts
  Lemma~\ref{lem:infty}, so no absolutely continuous invariant
  probability measure can exist.
\end{proof}
 

\subsection*{Acknowledgements} We would like to thank
J. Rivera-Letelier and the referee for helpful suggestions that
improved the final version of the paper.

\bibliography{bibl}
\bibliographystyle{plain}

\end{document}